\documentclass[fleqn,final,draft]{zzz}
\usepackage{mathrsfs}
\usepackage{showkeys} 
\usepackage{color}
\usepackage{tikz}
\usepackage{graphicx}
\usepackage{rotating}
\usepackage[pdftex]{hyperref}
\usepackage{soul}
\usepackage{subfig}
\usepackage{mathrsfs}
\usepackage{amsmath}
\usepackage{xparse}
\usepackage{chngcntr}
\usepackage{apptools}
\usepackage{rotating}
\usepackage[normalem]{ulem}
\usepackage{float}
\usepackage{arydshln}
\usepackage{pdfpages}

\DeclareMathAlphabet{\mathpzc}{OT1}{pzc}{m}{it}

\hypersetup{
 colorlinks = true,
 urlcolor = blue,
 linkcolor = red,
 citecolor = green
}

\definecolor{Mycolor2}{HTML}{aa2649}
\definecolor{Mycolor1}{HTML}{7b2cbf}
\definecolor{Mycolor3}{HTML}{ff6d00}

\sethlcolor{black}
\newcommand\moro[1]{{\textcolor{blue}{#1}}} 

\newcommand\boro[1]{{\textcolor{black}{#1}}}

\newcommand{\hyp}[5]{\,\mbox{}_{#1}F_{#2}\!\left(
 \genfrac{}{}{0pt}{}{#3}{#4};#5\right)}

\newcommand{\qhyp}[5]{\,\mbox{}_{#1}\phi_{#2}\!\left(
\genfrac{}{}{0pt}{}{#3}{#4};#5\right)}
\newcommand{\Whyp}[5]{\,\mbox{}_{#1}W_{#2}\!\left({#3};{#4};{#5}\right)}

\def\cprime{$'$}

\newtheorem{thm}{Theorem}[section]
\newtheorem{cor}[thm]{Corollary}

\newtheorem{rem}[thm]{Remark}
\newtheorem{defn}[thm]{Definition}

\makeatletter
\def\eqnarray{\stepcounter{equation}\let\@currentlabel=\theequation
\global\@eqnswtrue
\tabskip\@centering\let\\=\@eqncr
$$\halign to \displaywidth\bgroup\hfil\global\@eqcnt\z@
 $\displaystyle\tabskip\z@{##}$&\global\@eqcnt\@ne
 \hfil$\displaystyle{{}##{}}$\hfil
 &\global\@eqcnt\tw@ $\displaystyle{##}$\hfil
 \tabskip\@centering&\llap{##}\tabskip\z@\cr}

\def\endeqnarray{\@@eqncr\egroup
 \global\advance\c@equation\m@ne$$\global\@ignoretrue}

\def\@yeqncr{\@ifnextchar [{\@xeqncr}{\@xeqncr[5pt]}}
\makeatother

\newcommand{\expe}{{\mathrm e}}

\newcommand{\SSS}{{\mathcal S}}

\newcommand{\CC}{{\mathbb C}}
\newcommand{\CCast}{{{\mathbb C}^\ast}}
\newcommand{\CCdag}{{{\mathbb C}^\dag}}

\newcommand{\N}{{\mathbb N}}

\newcommand{\C}{\mathbb{C}} 

\newcommand{\topt}[2]{\left\{\substack{{#1}\\{#2}}\right\}}
\newcommand{\tops}[3]{\left\{\substack{{#1}\\{#2}\\{#3}}\right\}}

\newcommand{\topss}[3]{\left\{\substack{{#1}\\[-0.12cm]{#2}\\{#3}}\right\}}

\begin{document}
\renewcommand{\PaperNumber}{***}
\FirstPageHeading

\ArticleName{Special values for
continuous $q$-Jacobi polynomials}

\ShortArticleName{Special values for
the continuous $q$-Jacobi polynomials
}

\Author{%
Howard S.~Cohl\,$^\dag\!\!\ $ and
Roberto S.~Costas-Santos\,$^\S\!\!\ $ 
}
\AuthorNameForHeading{
H.~S.~Cohl and R. S.~Costas-Santos 
}



\Address{$^\dag$~Applied and Computational Mathematics Division,
National Institute of Standards and Technology,
Gaithersburg, MD 20899-8910, USA
\URLaddressD{
\href{http://www.nist.gov/itl/math/msg/howard-s-cohl.cfm}
{http://www.nist.gov/itl/math/msg/howard-s-cohl.cfm}
}
} 
\EmailD{howard.cohl@nist.gov} 
\Address{$^\S$ Departamento de M\'etodos 
Cuantitativos, Universidad Loyola Andaluc\'ia, 
E-41704, Dos Hermanas, Seville, Spain
} 
\URLaddressD{
\href{http://www.rscosan.com}
{http://www.rscosan.com}
}
\EmailD{rscosa@gmail.com} 
\ArticleDates{Received ???, in final form ????; Published online ????}
\Abstract{%
We study special values for the continuous $q$-Jacobi polynomials and present applications of these special values which arise from bilinear generating functions, and in particular the Poisson kernel for these polynomials.}


\section{Preliminaries}

We adopt the following set 
notations: $\mathbb N_0:=\{0\}\cup\N=\{0, 1, 2, ...\}$, and we 
use the sets $\mathbb Z$, $\mathbb R$, $\mathbb C$ which represent 
the integers, real numbers and 
complex numbers respectively, $\CCast:=\CC\setminus\{0\}$, and 
$\CCdag:=\CCast\setminus \{z\in\CC: |z|=1\}$.
We also adopt the following {\it multiset} notation and conventions.
Let ${\bf a}:=\{a,b,c,d\}$, $a,b,c,d\in\mathbb \CCast$.
Throughout the paper, we assume that the empty sum 
vanishes and the 
empty product 
is unity.
\begin{defn} \label{def:2.1}
We adopt the following conventions for succinctly 
writing elements of lists. To indicate sequential positive and negative 
elements, we write
\[
\pm a:=\{a,-a\}.
\]
\noindent We also adopt the analogous notations
\[
\expe^{\pm i\theta}:=\{\expe^{i\theta},\expe^{-i\theta}\},\quad
z^{\pm}:=\{z,z^{-1}\}.
\]
Within a list of items, we define
\[
a+\topss{x_1}{\vdots}{x_n}:=\{a+x_1,\ldots,a+x_n\}.
\]
\noindent In the same vein, consider the numbers $f_s\in\mathbb C$ 
with $s\in{\mathcal S}\subset \N$,
with ${\mathcal S}$ finite.
Then, the notation
$\{f_s\}$
represents the multiset of all complex numbers $f_s$ such that 
$s\in\SSS$.
Furthermore, consider some $p\in\SSS$, then the notation
$\{f_s\}_{s\ne p}$ represents the sequence of all complex numbers
$f_s$ such that $s\in\SSS\!\setminus\!\{p\}$.
In addition, for the empty list, $n=0$, we take
\[
\{a_1,{...},a_n\}:=\emptyset.
\]
\end{defn}
\noindent In this paper we will be using \moro{standard notations for both finite and $q$-shifted factorials:~including multi-$q$-shifted factorial notation where a comma delineated list represents products (see \cite[Appendix I]{GaspRah}). We also adopt standard notations for} both terminating and nonterminating basic hypergeometric functions ${}_r\phi_s$ \cite[(17.4.1)]{NIST:DLMF}, and a nice specialization which is often referred to as very-well-poised basic hypergeometric functions ${}_{r+1}W_r$ \cite[(2.1.11)]{GaspRah}. For those not familiar with these functions, we urge the reader to refer to some important standard references on the subject, and in particular \cite[Chapter 17]{NIST:DLMF} and links therein and the entire  book by Gasper \& Rahman (2004) \cite{GaspRah}. For particulars relating to orthogonal polynomials we suggest the reader to refer to \cite[Chapters 1,9,14]{Koekoeketal}, the monograph \cite{Ismail:2009:CQO} and the Memoirs of AMS article by Askey \& Wilson \cite{AskeyWilson85}. 
For a general treatment of special functions, one should refer to \cite{AAR} and \cite{NIST:DLMF}.

\subsection{The Askey--Wilson polynomials}

The Askey--Wilson polynomials can be defined in 
terms of the terminating balanced basic hypergeometric series
\cite[(14.1.1)]{Koekoeketal}
\begin{equation}
p_n(x;{\bf a}|q):=a^{-n}(ab,ac,ad;q)_n\qhyp43{q^{-n},q^{n-1}abcd,az^\pm
}{ab,ac,ad}{q,q},
\label{AW}
\end{equation}
where $x=\tfrac12(z+z^{-1})$.
In some of the derivations given below,
we use the renormalized version of the 
Askey--Wilson polynomials given by 
\begin{equation}
r_n(x;{\bf a}|q):=\qhyp43{q^{-n},q^{n-1}abcd,az^\pm
}
{ab,ac,ad}{q,q}=\frac{a^n}{(ab,ac,ad;q)_n}\, 
p_n(x;{\bf a}|q).
\label{RAW}
\end{equation}
The Askey--Wilson polynomials have the following
special values
\cite[(114)]{KoornwinderKLSadd}
\begin{equation}
\label{specAW}
p_n(\tfrac12(a+a^{-1});a,b,c,d|q)=a^{-n}(ab,ac,ad;q)_n
\end{equation}
(and similarly for arguments $\tfrac12(b+b^{-1})$,
$\tfrac12(c+c^{-1})$,
$\tfrac12(d+d^{-1})$).

\subsection{The $q$-Racah polynomials}

Let $m\in\C$, $n,N\in\N_0$ 
such that $n\in\{0,\ldots,N\}$.
Let us consider the $q$-Racah polynomials $R_n(\mu(m);\bar{\alpha},\bar{\beta},\bar{\gamma},\bar{\delta}|q)$
which are discrete cases of the 
Askey--Wilson polynomials \eqref{AW}
(note that we have used the notation $\{\bar{\alpha},\bar{\beta},\bar{\gamma},\bar{\delta}\}$ instead of the standard
$\{\alpha,\beta,\gamma,\delta\}$ in order to
disambiguate between the 
$\{\alpha,\beta\}$ parameters 
which appear in the study of continuous $q$-Jacobi polynomials (see \S\ref{ctsqJac} below)).
The $q$-Racah polynomials are
defined by 
\cite[(14.2.1)]{Koekoeketal}
\begin{equation}
R_n(\mu(m);\bar{\alpha},\bar{\beta},\bar{\gamma},\bar{\delta}|q):=\qhyp43{q^{-n},q^{n+1}\bar{\alpha}\bar{\beta},q^{-m},q^{m+1}\bar{\gamma}\bar{\delta}}{q\bar{\alpha},q\bar{\beta}\bar{\delta},q\bar{\gamma}}{q,q},
\label{qRdefn}
\end{equation}
where
$\mu(m):=\mu(m;{\bar{\gamma}},{\bar{\delta}}|q):=q^{-m}+q^{m+1}\bar{\gamma}\bar{\delta}$,
and 
\begin{equation}
{\bar{\alpha}}=q^{-N-1}\quad \mbox{or}\quad {\bar{\beta}}{\bar{\delta}}=q^{-N-1}\quad \mbox{or}\quad 
{\bar{\gamma}}=q^{-N-1}.
\label{qRcond}
\end{equation}
The $q$-Racah polynomials are 
orthogonal on the finite $q$-quadratic
set 
$\{\mu(m;{\bar{\gamma}},{\bar{\delta}}|q)\}$.
See \cite[(14.2.2)]{Koekoeketal} for the 
orthogonality relation for $q$-Racah polynomials.
Observe that for $m\in\N_0$, the $q$-Racah polynomials
satisfy the 
following duality relation \cite[(146)]{KoornwinderKLSadd}
\begin{equation}\label{eq:duality}
R_n(\mu(m);{\bar{\alpha}},{\bar{\beta}},{\bar{\gamma}},{\bar{\delta}}|q)=
R_m(\mu(n);{\bar{\gamma}},{\bar{\delta}},{\bar{\alpha}},{\bar{\beta}}|q). 
\end{equation}

\subsection{The continuous {\it q}-Jacobi polynomials}
\label{ctsqJac}
The continuous $q$-Jacobi polynomials
\cite[Section 14.10]{Koekoeketal} 
follow from the Askey--Wilson polynomials 
as follows 
\cite[(14.1.19) and p.~467]{Koekoeketal}
\begin{equation}
P_{n}^{(\alpha,\beta)} (x|q):=
\frac{q^{(\frac\alpha2+\frac14)n}}
{(q,-q^{\frac{\alpha+\beta}{2}+\frac12\topt{1}{2}}
;q)_n} p_n (x;{\bf a}|q),
\label{inter:AWqJac}
\end{equation}
where
\begin{equation}
{\bf a}:=\{a,b,c,d\}:=
\left\{q^{\frac12\alpha+\frac14\topt{1}{3}},
-q^{\frac12\beta+\frac14\topt{1}{3}}\right\}.
\label{abcd}
\end{equation}
For the coefficients $\{a,b\}=\{\pm q^\frac12\}$, 
$c={q^{\alpha+\frac12}}$, $d=-{q^{\beta+\frac12}}$, the 
renormalized Askey--Wilson polynomials are related to the continuous 
$q$-Jacobi polynomials with base $q^2$, namely
\begin{equation}
r_n(x;\pm q^\frac12,{q^{\alpha+\frac12}},-{q^{\beta+\frac12}}|q)=
q^{-n\alpha}
\frac{(q,-q^{\alpha+\beta+1};q)_n}{
(q^{\alpha+1},-q^{\beta+1};q)_n}
P_n^{(\alpha,\beta)}(x|{q^2}).
\end{equation}
In fact, this Askey--Wilson polynomial is related
by a quadratic transformation which follows from a
formula due to Singh
\cite[(4.22)]{AskeyWilson85},
\cite[(III.21)]{GaspRah},
\cite[(7)]{Singh59} 
\begin{equation}
\qhyp43{q^{-2n},q^{2n}a^2,c^2,qb^2}{-a,-qa,q^2b^2c^2}{q^2,q^2}
=(bc)^n\frac{(-q,-\frac{a}{bc};q)_n}
{(-a,-qbc;q)_n}
\qhyp43{q^{-n},q^n{a},\frac{c}{b},\frac{qb}{c}}{-q,-\frac{a}{bc},qbc}{q,q},
\label{Singhquad}
\end{equation}
or equivalently,
\begin{equation}
\qhyp43{q^{-n},q^{n}a^2,c^2,q^\frac12b^2}{-a,-q^\frac12a,qb^2c^2}{q,q}
=(bc)^n\frac{(-q^\frac12,-\frac{a}{bc};q^\frac12)_n}
{(-a,-q^\frac12bc;q^\frac12)_n}
\qhyp43{q^{-\frac{n}{2}},q^{\frac{n}{2}}{a},\frac{c}{b},\frac{q^\frac12b}
{c}}{-q^\frac12,-\frac{a}{bc},q^\frac12bc}{q^\frac12,q^\frac12},
\label{Singhquad2}
\end{equation}
but was also
proved independently by Askey and Wilson in \cite[(3.2)]{AskeyWilson85}),
\cite[(1.12)]{Rahman86prodctsqJ}
\begin{equation}
\hspace{-0.5cm}r_n\left(x;\pm q^{\frac12},{q^{\alpha+\frac12}},
-{q^{\beta+\frac12}}|q\right)
=q^{-n\alpha}
\frac{(-q^{\alpha+1},-q^{\alpha+\beta+1};q)_n}
{(-q^{\beta+1},-q;q)_n}
r_n\left(x;q^{\alpha+\frac12\topt{1}{3}},-q^{\beta+\frac12\topt{1}{3}}|q^2\right).
\label{ctsJacquad}
\end{equation}
The parity relation for continuous $q$-Jacobi polynomials is 
given by \cite[(165)]{KoornwinderKLSadd}
\begin{equation}
\label{ctsqJsymmetry}
P_n^{(\alpha,\beta)}(-x|q)=-q^{\frac12 n(\alpha-\beta)}
P_n^{(\beta,\alpha)}(x|q).
\end{equation}
The continuous $q$-Jacobi polynomials are orthogonal 
over $x=\frac12(\expe^{i\theta}+\expe^{-i\theta})=\cos\theta\in(-1,1)$ 
which is demonstrated by the orthogonality relation
\begin{equation}
\int_{-1}^1 P_m^{(\alpha,\beta)}(x|q)P_n^{(\alpha,\beta)}(x|q)\,
\frac{w(x;\alpha,\beta|q)}{\sqrt{1-x^2}}\,{\mathrm d}x=h_n(\alpha,\beta|q)
\delta_{m,n},
\label{octsqJac}
\end{equation}
where the weight function and the norm for the continuous 
$q$-Jacobi polynomials are given by
\begin{equation}
w(x;\alpha,\beta|q)
=\frac{(\expe^{\pm 2i\theta};q)_\infty}
{(q^{\frac12\alpha+\frac14}\expe^{\pm i\theta},
q^{\frac12\alpha+\frac34}\expe^{\pm i\theta},
-q^{\frac12\beta+\frac14}\expe^{\pm i\theta},
-q^{\frac12\beta+\frac34}\expe^{\pm i\theta} ;q)_\infty},
\label{wctsqJac}
\end{equation}
\begin{equation}
h_n(\alpha,\beta|q)=\frac{2\pi q^{(\alpha+\frac12)n}(q^{\frac{\alpha
+\beta+2}{2}},q^{\frac{\alpha+\beta+3}{2}};q)_\infty 
(q^{\alpha+1},q^{\beta+1},q^{\frac{\alpha+\beta+1}{2}};q)_n
}
{(q,q^{\alpha+1},q^{\beta+1},-q^{\frac{\alpha+\beta+1}{2}},
-q^{\frac{\alpha+\beta+2}{2}}
;q)_\infty(q,q^{\alpha+\beta+1},q^{\frac{\alpha+\beta+3}{2}};q)_n
}.
\label{normctsqJac}
\end{equation}

\begin{thm}
Let $n\in{\mathbb N}_0$, $|q|<1$, $\alpha,\beta\in{\mathbb C}$, 
$x=\frac12(z+z^{-1})\in\CC$. 
Then the exhaustive list of all balanced 
${}_4\phi_3$ continuous $q$-Jacobi polynomial representations 
are given by
\begin{eqnarray}
\label{ctsqJ:def1}&& \hspace{-1.40cm}P_n^{(\alpha,\beta)}(x|q) = 
\frac{(q^{\alpha+1};q)_n}{(q;q)_n}
\qhyp43{q^{-n},q^{\alpha+\beta+n+1}, q^{\frac12\alpha+\frac14}
z^{\pm}
}
{q^{\alpha+1},-q^{\frac{\alpha+\beta}{2}+\frac12\topt{1}{2}}}
{q,q} \\
\label{ctsqJ:def1b} &&\hspace{0.5cm}=
q^{-\frac{n}{2}}\frac{(q^{\alpha+1},-q^{\frac{\alpha+\beta+3}{2}};q)_n}
{(q,-q^{\frac{\alpha+\beta+1}{2}};q)_n}
\qhyp43{q^{-n},q^{\alpha+\beta+n+1}, q^{\frac12\alpha+\frac34}z^\pm }
{q^{\alpha+1},-q^{\frac{\alpha+\beta}{2}+\frac12\topt{2}{3}}}
{q,q} 
\\
\label{ctsqJ:def1c} &&\hspace{0.5cm}=
\left(-q^{\frac{\alpha-\beta}{2}}\right)^n\frac{(q^{\beta+1};q)_n}
{(q;q)_n}
\qhyp43{q^{-n},q^{\alpha+\beta+n+1}, -q^{\frac12\beta+\frac14}z^\pm }
{q^{\beta+1},-q^{\frac{\alpha+\beta}{2}+\frac12\topt{1}{2}}}
{q,q} 
\\
\label{ctsqJ:def1d} &&\hspace{0.5cm}=
\left(-q^{\frac{\alpha-\beta-1}{2}}\right)^n
\frac{(q^{\beta+1},-q^{\frac{\alpha+\beta+3}{2}};q)_n}
{(q,-q^{\frac{\alpha+\beta+1}{2}};q)_n}
\qhyp43{q^{-n},q^{\alpha+\beta+n+1}, -q^{\frac12\beta+\frac34}z^\pm }
{q^{\beta+1},-q^{\frac{\alpha+\beta}{2}+\frac12\topt{2}{3}}}
{q,q};
\end{eqnarray}
\begin{eqnarray}
&&\hspace{-0.60cm}P_n^{(\alpha,\beta)}(x|q)\nonumber \\
&&\hspace{0.0cm}=
q^{-\binom{n}{2}}\left(-q^{-\frac12}\right)^n
\frac{
\left(q^{\frac{\alpha+\beta}{2}+\frac12\topt{1}{2}},q^{\frac12\alpha+\frac34} 
z^\pm ;q\right)_n}{\left(q,q^{\alpha+\beta+1};q\right)_n}
\qhyp43{q^{-n}, q^{-\alpha-n},
-q^{-\frac{\alpha+\beta}{2}-n-\frac12\topt{0}{1}}
}
{q^{-\alpha-\beta-2n}, 
q^{-\frac12\alpha+\frac14-n}z^\pm}{q,q}\label{ctsqJ:def2b}\\
&&\hspace{0.0cm}=
q^{-\binom{n}{2}}(-1)^n
\frac{
\left(q^{\frac{\alpha+\beta}{2}+\frac12\topt{1}{2}},q^{\frac12\alpha+\frac14} 
z^\pm ;q\right)_n}{\left(q,q^{\alpha+\beta+1};q\right)_n}
\qhyp43{q^{-n}, q^{-\alpha-n},
-q^{-\frac{\alpha+\beta}{2}-n+\frac12\topt{0}{1}}
}
{q^{-\alpha-\beta-2n}, 
q^{-\frac12\alpha+\frac34-n}z^\pm}{q,q}\label{ctsqJ:def2} \\
&&\hspace{0.0cm}=
\left(q^{\frac{\alpha-\beta-n}2}\right)^n
\frac{
\left(q^{\frac{\alpha+\beta}{2}+\frac12\topt{1}{2}},-q^{\frac12\beta+\frac34} 
z^\pm ;q\right)_n}{\left(q,q^{\alpha+\beta+1};q\right)_n}
\qhyp43{q^{-n}, q^{-\beta-n},
-q^{-\frac{\alpha+\beta}{2}-n-\frac12\topt{0}{1}}
}
{q^{-\alpha-\beta-2n}, 
-q^{-\frac12\beta+\frac14-n}z^\pm }
{q,q} \label{ctsqJ:def2d} \\
&&\hspace{0.0cm}=
q^{-\binom{n}{2}}\left(q^{\frac{\alpha-\beta}2}\right)^n
\frac{
\left(q^{\frac{\alpha+\beta}{2}+\frac12\topt{1}{2}},-q^{\frac12\beta+\frac14} 
z^\pm ;q\right)_n}{\left(q,q^{\alpha+\beta+1};q\right)_n}
\qhyp43{q^{-n}, q^{-\beta-n},
-q^{-\frac{\alpha+\beta}{2}-n+\frac12\topt{0}{1}}
}
{q^{-\alpha-\beta-2n}, 
-q^{-\frac12\beta+\frac34-n}z^\pm }
{q,q}\!;\label{ctsqJ:def2c}
\end{eqnarray}
and
\begin{eqnarray}
&&\hspace{-0.3cm}
P_n^{(\alpha,\beta)}(x|q)
= z^n q^{(\frac\alpha2+\frac14)n} 
\frac{\left(q^{\alpha+1},-q^{\frac12\beta+\frac14\topt{1}{3}}z^{-1};q\right)_n}
{(q,-q^{\frac{\alpha+\beta}{2}+\frac12\topt{1}{2}} ;q)_n}
\qhyp43{q^{-n},q^{-\beta-n},
q^{\frac12\alpha+\frac14\topt{1}{3}}z}
{q^{\alpha+1},-q^{-\frac12\beta-n+\frac14\topt{1}{3}}z}{q,q}\label{ctsqJ:def3} \\
&&\hspace{0.1cm}= z^n q^{(\frac\alpha2+\frac14)n} 
\frac{\left(q^{\beta+1},
q^{\frac12\alpha+\frac14\topt{1}{3}}z^{-1};q\right)_n}
{(q,-q^{\frac{\alpha+\beta}{2}+\frac12\topt{1}{2}} ;q)_n}
\qhyp43{q^{-n},q^{-\alpha-n},
-q^{\frac12\beta+\frac14\topt{1}{3}}z}
{q^{\beta+1},q^{-\frac12\alpha-n+\frac14\topt{1}{3}
}z
}{q,q} \label{ctsqJ:def3f}\\
 &&\hspace{0.1cm}= z^n q^{(\frac\alpha2 +\frac14)n} 
\frac{\left(q^{\frac\alpha2+\frac34}z^{-1},
-q^{\frac\beta2+\frac34}z^{-1};q\right)_n}
{(q,-q^{\frac{\alpha+\beta+2}{2}} ;q)_n}
\qhyp43{q^{-n},-q^{-\frac{\alpha+\beta+1}{2}-n},
q^{\frac12\alpha+\frac14}z,
-q^{\frac12\beta+\frac14}z}
{-q^{\frac{\alpha+\beta+1}{2}},q^{-\frac12\alpha+\frac14-n}z,
-q^{-\frac12\beta+\frac14-n}z
}{q,q}\label{ctsqJ:def3b}\\
 &&\hspace{0.1cm}=z^n q^{(\frac\alpha2+\frac14)n} 
\frac{\left(q^{\frac\alpha2+\frac34}z^{-1},
-q^{\frac\beta2+\frac14}z^{-1};q\right)_n}
{(q,-q^{\frac{\alpha+\beta+1}{2}} ;q)_n}
\qhyp43{q^{-n},-q^{-\frac{\alpha+\beta}{2}-n},
q^{\frac12\alpha+\frac14}z,
-q^{\frac12\beta+\frac34}z}
{-q^{\frac{\alpha+\beta+2}{2}},
q^{-\frac12\alpha+\frac14-n}z,
-q^{-\frac12\beta+\frac34-n}z
}{q,q} \label{ctsqJ:def3c}\\
 &&\hspace{0.1cm}=z^n q^{(\frac\alpha2+\frac14)n} 
\frac{\left(q^{\frac\alpha2+\frac14}z^{-1},
-q^{\frac\beta2+\frac34}z^{-1};q\right)_n}
{(q,-q^{\frac{\alpha+\beta+1}{2}} ;q)_n}
\qhyp43{q^{-n},-q^{-\frac{\alpha+\beta}{2}-n},
q^{\frac12\alpha+\frac34}z,
-q^{\frac12\beta+\frac14}z}
{-q^{\frac{\alpha+\beta+2}{2}},
q^{-\frac12\alpha+\frac34-n}z,
-q^{-\frac12\beta+\frac14-n}z
}{q,q} \label{ctsqJ:def3d} \\
\nonumber &&\hspace{0.1cm}=z^n q^{(\frac\alpha2+\frac14)n}
\frac{\left(-q^{\frac{\alpha+\beta+3}{2}},
q^{\frac\alpha2+\frac14}z^{-1},
-q^{\frac\beta2+\frac14}z^{-1};q\right)_n}
{(q,-q^{\frac{\alpha+\beta+1}{2}},-q^{\frac{\alpha+\beta+2}{2}} ;q)_n}\\
 &&\hspace{6.45cm}\times
\qhyp43{q^{-n},-q^{-\frac{\alpha+\beta-1}{2}-n},
q^{\frac\alpha2+\frac34}z,
-q^{\frac\beta2+\frac34}z}
{-q^{\frac{\alpha+\beta+3}{2}},
q^{-\frac\alpha2+\frac34-n}z,
-q^{-\frac\beta2+\frac34-n}z}
{q,q},\label{ctsqJ:def3e}
\end{eqnarray}
and as well as the application of the map 
$z\mapsto z^{-1}$ in
\eqref{ctsqJ:def3}--\eqref{ctsqJ:def3e}. 
\end{thm}
\medskip
\begin{proof}
In \cite[Theorem 7, (13)--(15)]{CohlCostasSantos20b} an exhaustive 
description of all balanced ${}_4\phi_3$
representations of the Askey--Wilson polynomials
are given with parameters ${\bf a}:=
\{a_1,a_2,a_3,a_4\}$. By using \eqref{inter:AWqJac}
which is the description of the continuous $q$-Jacobi polynomials 
in terms of the Askey--Wilson polynomials, one is able to obtain an 
exhaustive list of all balanced ${}_4\phi_3$ representations of the 
continuous $q$-Jacobi polynomials. In line with \eqref{inter:AWqJac}, 
we choose 
\begin{equation}
{\bf a}:=\left\{q^{\frac12\alpha+\frac14}
,q^{\frac12\alpha+\frac34},-q^{\frac12\beta+\frac14}
,-q^{\frac12\beta+\frac34}\right\}.
\label{cqJAWparms}
\end{equation}
By evaluating all permutations of 
\eqref{cqJAWparms} in \cite[(13), (14), (15)]{CohlCostasSantos20b}, 
one may obtain \eqref{ctsqJ:def1}--\eqref{ctsqJ:def1d}, 
\eqref{ctsqJ:def2d}--\eqref{ctsqJ:def2c}, 
\eqref{ctsqJ:def3}--\eqref{ctsqJ:def3e} 
respectively. Since the continuous $q$-Jacobi polynomials are a 
function of $x=\frac12(z+z^{-1})$, and 
\eqref{ctsqJ:def1}--\eqref{ctsqJ:def2c}
are invariant under the interchange of $z\mapsto z^{-1}$, one 
may as well apply this replacement to \eqref{ctsqJ:def3}--\eqref{ctsqJ:def3e} 
to obtain six alternative representations of the continuous 
$q$-Jacobi polynomials. This completes the proof.
\end{proof}
\begin{rem}
Note that one may also obtain an exhaustive list of all 
continuous $q$-Jacobi polynomial terminating ${}_8W_7$ 
representations by applying the above procedure to 
\cite[(16)--(19)]{CohlCostasSantos20b}. However, we will omit
this computation for the present work. We might also add that 
by counting the members of the equivalence classes of 
terminating ${}_4\phi_3$ representations of the continuous 
$q$-Jacobi polynomials, one may explore the symmetry group of 
these transformations which should necessarily be a subgroup
of the symmetric group $S_6$, the symmetry group of the 
terminating representations of the Askey--Wilson polynomials, 
see \cite[Section 4]{CohlCostasSantos20b}, and references therein.
\end{rem}
\section{Specialization values of continuous $q$-Jacobi polynomials}
\label{SpecctsqJac}
\noindent
For the following special values for the argument of the continuous 
$q$-Jacobi polynomials, we are able to re-express them in terms
of Askey--Wilson polynomials with degree $m$
as opposed to $n$.

\begin{thm}
\label{thmspecab}
Let $n,m\in\N_0$, $q\in\CCdag$, $\alpha,\beta\in\C$. 
The following 
continuous $q$-Jacobi polynomial specializations
have alternative Askey--Wilson representations given as follows
\begin{eqnarray}
&&\label{sp1}\hspace{0.0cm}P_n^{(\alpha,\beta)}(\tfrac12(q^{\frac12\alpha
+\frac14+m}+q^{-\frac12\alpha-\frac14-m})|q)
=\frac{
(q^{\alpha+1};q)_n\left(q^{\frac12(\alpha+\beta+1)}\right)^m}
{(q;q)_n(q^{\alpha+1},-q^{\frac12(\alpha+\beta+1)},
-q^{\frac12(\alpha+\beta+2)};q)_m}\nonumber\\
&&\hspace{3.5cm}\times p_m(\tfrac12(q^{\frac{\alpha+\beta+1}2+n}
+q^{-\frac{\alpha+\beta+1}2-n});q^{\frac12(\alpha+\beta+1)},
q^{\frac12(\alpha-\beta+1)},-q^\frac12,-1|q),\\
&&\label{sp2}\hspace{0.0cm}P_n^{(\alpha,\beta)}(\tfrac12(q^{\frac12\alpha
+\frac34+m}+q^{-\frac12\alpha-\frac34-m})|q)
=
\frac{
q^{-\frac n2}\left(q^{\frac12(\alpha+\beta+1)}\right)^m (q^{\alpha+1},
-q^{\frac{\alpha+\beta+3}{2}};q)_n}
{(q,-q^{\frac{\alpha+\beta+1}{2}};q)_n(q^{\alpha+1},-q^{\frac12(\alpha+\beta+2)},
-q^{\frac12(\alpha+\beta+3)};q)_m}\nonumber\\
&&\hspace{3.5cm}\times p_m(\tfrac12(q^{\frac{\alpha+\beta+1}2+n}
+q^{-\frac{\alpha+\beta+1}2-n});
q^{\frac12(\alpha+\beta+1)},q^{\frac12(\alpha-\beta+1)},-q^\frac12,-q|q),\\
&&\hspace{0.0cm}\label{sp3}P_n^{(\alpha,\beta)}
(-\tfrac12(q^{\frac12\beta+\frac14+m}
+q^{-\frac12\beta-\frac14-m})|q)
=
\frac{\left(-q^{\frac12(\alpha-\beta)}\right)^n
\left(q^{\frac12(\alpha+\beta+1)}\right)^m(q^{\beta+1};q)_n}
{(q;q)_n(q^{\beta+1},-q^{\frac12(\alpha+\beta+1)},
-q^{\frac12(\alpha+\beta+2)};q)_m}\nonumber\\
&&\hspace{3.5cm}\times p_m(\tfrac12(q^{\frac{\alpha+\beta+1}2+n}
+q^{-\frac{\alpha+\beta+1}2-n});
q^{\frac12(\alpha+\beta+1)},q^{\frac12(\beta-\alpha+1)},-q^\frac12,-1|q),\\
&&\hspace{0.0cm}\label{sp4}P_n^{(\alpha,\beta)}
(-\tfrac12(q^{\frac12\beta+\frac34+m}+q^{-\frac12\beta-\frac34-m})|q)
=\frac{\left(-q^{\frac12(\alpha-\beta-1)}\right)^n
\left(q^{\frac12(\alpha+\beta+1)}\right)^m
(q^{\beta+1},-q^{\frac12(\alpha+\beta+3)};q)_n}
{(q,-q^{\frac12(\alpha+\beta+1)};q)_n(q^{\beta+1},-q^{\frac12(\alpha+\beta+2)},
-q^{\frac12(\alpha+\beta+3)};q)_m}\nonumber\\
&&\hspace{3.5cm}\times p_m(\tfrac12(q^{\frac{\alpha+\beta+1}2+n}
+q^{-\frac{\alpha+\beta+1}2-n});
q^{\frac12(\alpha+\beta+1)},q^{\frac12(\beta-\alpha+1)},-q^\frac12,-q|q).
\end{eqnarray}
\end{thm}
\begin{proof}
Start by considering the Askey--Wilson 
polynomial representations of these continuous
$q$-Jacobi polynomials 
\eqref{inter:AWqJac}
with these particular arguments. The specific
arguments provided along with the 
identification that $x=\frac12(z+z^{-1})$ 
and therefore
\begin{equation}
z=z_m\in\{
q^{\frac12\alpha+\frac14+m},
q^{\frac12\alpha+\frac34+m},
-q^{\frac12\beta+\frac14+m},
-q^{\frac12\beta+\frac34+m}\},
\end{equation}
for \eqref{sp1}--\eqref{sp4} respectively. Providing the particular 
specializations of the argument produces the following ${}_4\phi_3$ 
representations
\begin{eqnarray}
&&\hspace{-0.6cm}P_n^{(\alpha,\beta)}(\tfrac12(q^{\frac\alpha2
+\frac14+m}+q^{-\frac\alpha2-\frac14-m})|q)
=\frac{(q^{\alpha+1};q)_n}{(q;q)_n}
\qhyp43{q^{-n},q^{\alpha+\beta+1+n},q^{-m},q^{\alpha+\frac12+m}}
{q^{\alpha+1},-q^{\frac{\alpha+\beta+1}{2}},
-q^{\frac{\alpha+\beta+2}{2}}}{q,q},\label{qhnm1}\\
&&\hspace{-0.6cm}P_n^{(\alpha,\beta)}(\tfrac12(q^{\frac\alpha2+\frac34+m}
+q^{-\frac\alpha2-\frac34-m})|q)
\nonumber\\&&\hspace{0.3cm}
=q^{-\frac n2}\frac{(q^{\alpha+1},-q^{\frac{\alpha+\beta+3}{2}};q)_n}
{(q,-q^{\frac{\alpha+\beta+1}{2}};q)_n}
\qhyp43{q^{-n},q^{\alpha+\beta+1+n},q^{-m},q^{\alpha+\frac32+m}}
{q^{\alpha+1},-q^{\frac{\alpha+\beta+2}{2}},
-q^{\frac{\alpha+\beta+3}{2}}}{q,q},\label{qhnm2}\\
&&\hspace{-0.6cm}P_n^{(\alpha,\beta)}(-\tfrac12(q^{\frac\beta2+\frac14+m}
+q^{-\frac\beta2-\frac14-m})|q)
\nonumber\\&&\hspace{0.3cm}=\left(-q^{\frac{\alpha-\beta}{2}}\right)^n
\frac{(q^{\beta+1};q)_n}{(q;q)_n}
\qhyp43{q^{-n},q^{\alpha+\beta+1+n},q^{-m},q^{m+\beta+\frac 12}}
{q^{\beta+1},-q^{\frac{\alpha+\beta+1}{2}},
-q^{\frac{\alpha+\beta+2}{2}}}{q,q},\label{qhnm3}\\
&&\hspace{-0.6cm}P_n^{(\alpha,\beta)}(-\tfrac12(q^{\frac\beta2
+\frac34+m}+q^{-\frac\beta2-\frac34-m})|q)\nonumber\\&&\hspace{0.3cm}
=\left(-q^{\frac{\alpha-\beta-1}{2}}\right)^n\frac{(q^{\beta+1},
-q^{\frac{\alpha+\beta+3}{2}};q)_n}{(q,-q^{\frac{\alpha+\beta+1}{2}};q)_n}
\qhyp43{q^{-n},q^{\alpha+\beta+1+n},q^{-m},q^{\beta+\frac 32+m}}
{q^{\beta+1},-q^{\frac{\alpha+\beta+2}{2}},
-q^{\frac{\alpha+\beta+3}{2}}}{q,q}.
\label{qhnm4}
\end{eqnarray}
It is then straightforward to convert these particular 
continuous $q$-Jacobi polynomials into
a form where they can be represented as either
Askey--Wilson polynomials in the degree $n$ but also as 
Askey--Wilson polynomials in degree $m$. Solving for the 
particular values of the constants ${\bf a}$ in 
\eqref{AW} for the degree $m$ case completes the proof.
\end{proof}
One, therefore, has the following special values for the continuous 
$q$-Jacobi polynomials which also follow directly from the identity
\eqref{specAW}.

\begin{cor}
\label{spec}
Let $n,m\in\N_0$, $q\in\CCdag$, $\alpha,\beta\in\C$. Then
\begin{eqnarray}
&&\hspace{-0.5cm}P_n^{(\alpha,\beta)}(\tfrac12(q^{\frac12\alpha
+\frac14}+q^{-\frac12\alpha-\frac14})|q)
=\frac{(q^{\alpha+1};q)_n}{(q;q)_n},
\label{spec1or}\\
&&\hspace{-0.5cm}P_n^{(\alpha,\beta)}(\tfrac12(q^{\frac12\alpha+\frac34}
+q^{-\frac12\alpha-\frac34})|q)
=q^{-\frac n2}\frac{(q^{\alpha+1},-q^{\frac12(\alpha+\beta+3)};q)_n}
{(q,-q^{\frac12(\alpha+\beta+1)};q)_n},\label{spec2}\\
&&\hspace{-0.5cm}P_n^{(\alpha,\beta)}(-\tfrac12(q^{\frac12\beta
+\frac14}+q^{-\frac12\beta-\frac14})|q)
=\left(-q^{\frac{\alpha-\beta}{2}}\right)^n
\frac{(q^{\beta+1};q)_n}{(q;q)_n},\label{spec3}\\
&&\hspace{-0.5cm}P_n^{(\alpha,\beta)}(-\tfrac12(q^{\frac12\beta
+\frac34}+q^{-\frac12\beta-\frac34})|q)
=\left(-q^{\frac{\alpha-\beta-1}{2}}\right)^n
\frac{(q^{\beta+1},-q^{\frac12(\alpha+\beta+3)};q)_n}
{(q,-q^{\frac12(\alpha+\beta+1)};q)_n}.
\label{spec4}
\end{eqnarray}
\end{cor}
\begin{proof}
Taking $m=0$ values in Theorem 
\ref{thmspecab} completes the proof.
\end{proof}
\begin{rem}
One may inquire regarding the computation of
perhaps product formulas for
\begin{eqnarray}
P_n^{(\alpha,\beta)}(-\tfrac12(q^{\frac12\alpha+\frac14}
+q^{-\frac12\alpha-\frac14})|q),
P_n^{(\alpha,\beta)}(-\tfrac12(q^{\frac12\alpha+\frac34}
+q^{-\frac12\alpha-\frac34})|q),\nonumber\\
P_n^{(\alpha,\beta)}(\tfrac12(q^{\frac12\beta+\frac14}
+q^{-\frac12\beta-\frac14})|q),
P_n^{(\alpha,\beta)}(\tfrac12(q^{\frac12\beta+\frac34}
+q^{-\frac12\beta-\frac34})|q),\nonumber
\end{eqnarray}
using
symmetry \eqref{ctsqJsymmetry}
and the special values for the Askey--Wilson polynomials given
\cite[(114)]{KoornwinderKLSadd}. However, this is not possible 
since the interchange of $\alpha\leftrightarrow\beta$ with the 
invariance of the argument prevents the ${}_4\phi_3$s 
from being summable.
\end{rem}

One may also express the results of
Theorem \ref{thmspecab} in terms of 
$q$-Racah polynomials
\eqref{qRdefn}.

\begin{thm}
Let $q\in\CCdag$. Then the specialization
for the continuous $q$-Jacobi polynomials
there are the following special values
given in terms of $q$-Racah polynomials
as follows 
\begin{eqnarray}
&&\hspace{-0.6cm}P_n^{(\alpha,\beta)}(\tfrac12(q^{\frac\alpha2+\frac14+m}
+q^{-\frac\alpha2-\frac14 m})|q)\nonumber\\&&\hspace{0.3cm}
=\frac{(q^{\alpha+1};q)_n}{(q;q)_n}
R_n(q^{\alpha+\frac12+m}+q^{-m};q^\alpha,q^\beta,
-q^{\frac{\alpha+\beta}{2}},-q^{\frac{\alpha-\beta-1}{2}}|q)\label{qRn1}\\
&&\hspace{0.3cm}=\frac{(q^{\alpha+1};q)_n}{
(q;q)_n}R_m(q^{\alpha+\beta+1+n}+q^{-n};-q^{\frac{\alpha+\beta}{2}},
-q^{\frac{\alpha-\beta-1}{2}},q^{\alpha},q^\beta|q),\label{qRm1}\\
&&\hspace{-0.6cm}P_n^{(\alpha,\beta)}(\tfrac12(q^{\frac\alpha2+\frac34+m}
+q^{-\frac\alpha2-\frac34-m})|q)\nonumber\\&&\hspace{0.3cm}
=q^{-\frac n2}\frac{(q^{\alpha+1},-q^{\frac{\alpha+\beta+3}{2}};q)_n}
{(q,-q^{\frac{\alpha+\beta+1}{2}};q)_n}
R_n(q^{\alpha+\frac32+m}+q^{-m};q^\alpha,q^\beta,
-q^{\frac{\alpha+\beta+1}{2}},-q^{\frac{\alpha-\beta}{2}}|q)\label{qRn2}\\
&&\hspace{0.3cm}=q^{-\frac n2}
\frac{(q^{\alpha+1},-q^{\frac{\alpha+\beta+3}{2}};q)_n}
{(q,-q^{\frac{\alpha+\beta+1}{2}};q)_n}
R_m(q^{\alpha+\beta+1+n}+q^{-n};-q^{\frac{\alpha+\beta+1}{2}},
-q^{\frac{\alpha-\beta}{2}},q^{\alpha},q^\beta|q),\label{qRm2}\\
&&\hspace{-0.6cm}P_n^{(\alpha,\beta)}(-\tfrac12(q^{\frac\beta2+\frac14+m}
+q^{-\frac\beta2-\frac14-m})|q)\nonumber\\
&&\hspace{0.3cm}=\left(-q^{\frac{\alpha-\beta}{2}}\right)^n
\frac{(q^{\beta+1};q)_n}{(q;q)_n}R_n(q^{\beta+\frac12+m}+q^{-m};
q^\beta, q^\alpha,-q^{\frac{\beta-\alpha-1}{2}},-q^{\frac{\alpha
+\beta}{2}}|q)\label{qRn3}\\
&&\hspace{0.3cm}=\left(-q^{\frac{\alpha-\beta}{2}}\right)^n
\frac{(q^{\beta+1};q)_n}{(q;q)_n}
R_m(q^{\alpha+\beta+1+n}+q^{-n};-q^{\frac{\alpha+\beta}{2}},
-q^{\frac{\beta-\alpha-1}{2}},q^{\beta},q^\alpha|q),\label{qRm3}\\
&&\hspace{-0.6cm}P_n^{(\alpha,\beta)}(-\tfrac12(q^{\frac\beta2+\frac34+m}
+q^{-\frac\beta2-\frac34-m})|q)\nonumber\\&&\hspace{0.3cm}
=\left(-q^{\frac{\alpha-\beta-1}{2}}\right)^n\frac{(q^{\beta+1},
-q^{\frac{\alpha+\beta+3}{2}};q)_n}{(q,-q^{\frac{\alpha+\beta+1}{2}};q)_n}
R_n(q^{\beta+\frac32+m}+q^{-m};q^\beta,q^\alpha,
-q^{\frac{\beta+\alpha+1}{2}},-q^{\frac{\beta-\alpha}{2}}|q)\label{qRn4}\\
&&\hspace{0.3cm}=\left(-q^{\frac{\alpha-\beta-1}{2}}\right)^n
\frac{(q^{\beta+1},-q^{\frac{\alpha+\beta+3}{2}};q)_n}
{(q,-q^{\frac{\alpha+\beta+1}{2}};q)_n}
R_m(q^{\alpha+\beta+1+n}+q^{-n};-q^{\frac{\alpha+\beta+1}{2}},
-q^{\frac{\beta-\beta}{2}},q^\beta,q^{\alpha}|q).\label{qRm4}
\end{eqnarray}
\label{dualthm1}
\end{thm}
\begin{proof}
Start with Theorem \ref{thmspecab} and utilize
\eqref{qhnm1}--\eqref{qhnm4} with \eqref{qRdefn} to
write these hypergeometric representations in 
terms of $q$-Racah polynomials. This completes the proof.
\end{proof}
\begin{cor}Let $m\in\N_0$. 
If $\beta=-\frac12$,
$\beta=\frac12$, $\alpha=-\frac12$, 
$\alpha=\frac12$, respectively for the following specialized
continuous $q$-Jacobi polynomials
\begin{eqnarray}
P_n^{(\alpha,\beta)}(\tfrac12(q^{\frac12\alpha+\frac14+m}
+q^{-\frac12\alpha-\frac14-m})|q),
P_n^{(\alpha,\beta)}(\tfrac12(q^{\frac12\alpha+\frac34+m}
+q^{-\frac12\alpha-\frac34-m})|q),\nonumber\\
P_n^{(\alpha,\beta)}(-\tfrac12(q^{\frac12\beta+\frac14+m}
+q^{-\frac12\beta-\frac14-m})|q),
P_n^{(\alpha,\beta)}(-\tfrac12(q^{\frac12\beta+\frac34+m}
+q^{-\frac12\beta-\frac34-m})|q),\nonumber
\end{eqnarray}
then one 
has the following duality relations
\begin{eqnarray}
&&\hspace{-1cm}R_n(q^{\alpha+\frac12+m}+q^{-m};q^{\alpha},
q^{-\frac12},-q^{\frac12\alpha\pm\frac14}|q)
=
R_m(q^{\alpha+\frac12+n}+q^{-n};q^{\alpha},
q^{-\frac12},-q^{\frac12\alpha\pm\frac14}|q).
\end{eqnarray}
\end{cor}
\begin{proof}
Replace $\beta=\pm \tfrac12$, $\alpha=\pm \tfrac12$ 
in Theorem \ref{dualthm1} completes the proof.
\end{proof}
Now consider the continuous $q$-Jacobi
polynomials with special argument
\begin{equation}
x_m^{\pm}:=\pm\tfrac12(q^{\frac14+\frac{m}2}\!+\!q^{-\frac14-\frac{m}2}), 
\label{xmpmdef}
\end{equation}
which is given in terms the Askey--Wilson polynomials through
\eqref{inter:AWqJac} as
\begin{eqnarray}
&&\hspace{-0.5cm}P_n^{(\alpha,\beta)}(x_m^{\pm}|q)
=\frac{\left(q^{\frac12\alpha+\frac14}\right)^n}
{(q,-q^{\frac{\alpha+\beta+1}{2}},-q^{\frac{\alpha+\beta+2}{2}};q)_n}
p_n(x_m^{\pm};
q^{\frac{\alpha}2+\frac14\topt{1}{3}},
-q^{\frac{\beta}2+\frac14\topt{1}{3}}
|q).
\label{ctsqJAWarg}
\end{eqnarray}
We consider some of the properties of the
continuous $q$-Jacobi polynomials with 
this argument. First, we show some 
alternative Askey--Wilson representations of these
polynomials with the special argument $x_m^{\pm}$.

\begin{thm}
Let $m,n\in\N_0$, $q\in\CCdag$, 
$\alpha,\beta\in\C$. 
Then
\begin{eqnarray}
&&\hspace{-0.4cm}
P_n^{(\alpha,\beta)}
(x_m^{\pm}|q)=
\frac{\left(\pm q^{\frac12\alpha+\frac14}\right)^n}{(q^\frac12,-q^\frac12,
-q^\frac{\alpha+\beta+1}{2};q^\frac12)_n}
p_n(x_m^{\pm};
q^{\frac14},
-q^{\frac14},
\pm q^{\frac{\alpha}2+\frac14},
\mp q^{\frac{\beta}2+\frac14}
|q^\frac12),\\
&&\hspace{0.7cm}=
\frac{
\left(\pm q^{\frac12\alpha}\right)^n
\left(q^{\frac{\alpha+\beta+1}{4}}\right)^m
(
\pm q^{\frac{\alpha+1}{2}},
\mp q^{\frac{\beta+1}{2}};q^\frac12)_n
}
{(q^\frac12,-q^\frac{\alpha+\beta+1}{2};q^\frac12)_n(-q^\frac12,\pm q^{\frac{\alpha+1}{2}},
\mp q^{\frac{\beta+1}{2}};q^\frac12)_m}\nonumber\\
&&\hspace{1.7cm}\times p_m(\tfrac12(q^{\frac{\alpha+\beta+1+2n}{4}}\!+\!q^{-\frac{\alpha+\beta+1+2n}4});
q^{\frac{\alpha+\beta+1}4},
-q^{\frac{-\alpha-\beta+1}4},
\pm q^{\frac{\alpha-\beta+1}4},
\mp q^{\frac{\beta-\alpha+1}4}
|q^\frac12).
\label{secondPmxm}
\end{eqnarray}
\label{thmTom}
\end{thm}
\begin{proof}
Start with the representation of the continuous $q$-Jacobi polynomials 
in terms of the Askey--Wilson polynomials
\eqref{inter:AWqJac} with special argument 
$x_m^{\pm}$ \eqref{xmpmdef}, namely \eqref{ctsqJAWarg}.
Now consider $x=\frac12(z+z^{-1})$, therefore
$z=z_m:=\pm q^{\frac14+\frac{m}{2}}$
in \eqref{ctsqJ:def1} and replacing $q\mapsto q^2$ one obtains,
\begin{eqnarray}
&&P_n^{(\alpha,\beta)}(\pm\tfrac12(q^{\frac12+{m}}\!+\!q^{-\frac12-{m}})|q^2)=
\frac{(q^{2\alpha+2};q^2)_n}{(q^2;q^2)_n}\nonumber\\
&&\hspace{5.5cm}\times
\qhyp43{q^{-2n},q^{2\alpha+2\beta+2+2n},\pm q^{\alpha+1+m},\pm 
q^{\alpha-m}}{q^{2\alpha+2},-q^{\alpha+\beta+1},
-q^{\alpha+\beta+2}}{q^2,q^2}. \label{qJacspec43}
\end{eqnarray}
Applying the quadratic transformation 
\eqref{Singhquad}
produces
\[
P_n^{(\alpha,\beta)}(\pm\tfrac12(q^{\frac12+{m}}+q^{-\frac12
-{m}})|q^2)\!=\!(\pm q^{\alpha})^n\frac{(\pm q^{\alpha+1},\mp q^{\beta+1};q)_n
}{(q,-q^{\alpha+\beta+1};q)_n}
\!\qhyp43{\!q^{-n},q^{\alpha+\beta+1+n}, q^{-m},q^{m+1}}{\pm q^{\alpha+1},
\mp q^{\beta+1},-q\!}{q,q\!}\!,
\]
which can be viewed as either an Askey--Wilson polynomial with 
degree $n$ or $m$.
Obtaining these representations through \eqref{AW},
and then replacing $q^2\mapsto q$, completes the proof.
\end{proof}
\begin{rem}
Note that the above proof could be accomplished directly 
using \eqref{Singhquad2} in \eqref{ctsqJ:def1}.
Furthermore one should observe that of all the 
${}_4\phi_3$ representations
\eqref{ctsqJ:def1}--\eqref{ctsqJ:def3e},
only \eqref{ctsqJ:def1} and \eqref{ctsqJ:def1c}
allow for the quadratic transformation \eqref{Singhquad} (or \eqref{Singhquad2}). 
However,
if one starts with \eqref{ctsqJ:def1c}
in order to prove Theorem \ref{thmTom},
one arrives at an identical result.
The fact that \eqref{ctsqJ:def1} and \eqref{ctsqJ:def1c}
are the only representations which satisfy the quadratic 
transformation \eqref{Singhquad} (or \eqref{Singhquad2})
can be seen since \eqref{ctsqJ:def2b}--\eqref{ctsqJ:def1c} 
do not contain the necessary 
$q^n$ numerator entry and \eqref{ctsqJ:def1b},
\eqref{ctsqJ:def1d} do not satisfy the conditions
given in \eqref{Singhquad2}.
\end{rem}

\noindent Now consider the $m=0,1$ special cases. 
This leads to the following result.
\begin{cor}
\label{specquad}
Let $n\in\N_0$, $q\in\CCdag$, $\alpha,\beta\in\C$. Then
\begin{eqnarray}
\label{spec1}
\hspace{-0.4cm}P_n^{(\alpha,\beta)}(\tfrac12(q^\frac12+q^{-\frac12})|q^2)&=&
\left(q^{{\alpha}}\right)^n
\frac
{(q^{{\alpha+1}},-q^{{\beta+1}};q)_n}
{( q,-q^{\alpha+\beta+1};q)_n},\\
\label{spec1b}
\hspace{-0.4cm}P_n^{(\alpha,\beta)}(-\tfrac12(q^\frac12+q^{-\frac12})|q^2)&=&
\left(- q^{\alpha}\right)^n
\frac
{( - q^{{\alpha+1}},q^{{\beta+1}};q)_n}
{( q,-q^{{\alpha+\beta+1}};q)_n},\\
\hspace{-0.4cm}P_n^{(\alpha,\beta)}(\tfrac12(q^{\alpha+\frac12}
+q^{-\alpha-\frac12}|q^2)&=&
\frac{(q^{\alpha+1},-q^{\alpha+1};q)_n} {(q,-q;q)_n},\\
\hspace{-0.4cm}P_n^{(\alpha,\beta)}(-\tfrac12(q^{\beta+\frac12}
+q^{-\beta-\frac12})|q^2)&=&
\left(-q^{\alpha-\beta}\right)^n
\frac{(q^{\beta+1},-q^{\beta+1};q)_n}
{(q,-q;q)_n}.
\end{eqnarray}
\end{cor}
\begin{proof}
For \eqref{spec1}, \eqref{spec1b} use the sum
\cite[(4.23)]{AskeyWilson85}
due to the quadratic transformation
for Askey--Wilson polynomials \eqref{Singhquad}.
\end{proof}
\begin{cor}
Let $n\in\N_0$, $q\in\CCdag$, $\alpha,\beta\in\C$. Then
one also has the following special 
value
\begin{eqnarray}
&&\hspace{-0.4cm}P_n^{(\alpha,\beta)}(\pm\tfrac12(q^\frac34+q^{-\frac34})|q)=
\left(\pm q^{\frac\alpha2}\right)^n
\frac
{(\pm q^{\frac{\alpha+1}{2}},\mp q^{\frac{\beta+1}{2}};q^\frac12)_n}
{( q^{\frac{1}{2}},-q^{\frac{\alpha+\beta+1}{2}};q^\frac12)_n}
\nonumber\\&&\hspace{5.2cm}\times
\left(1-\frac{(1-q^\frac12)(1-q^{-\frac{n}{2}})(1-q^{\frac{\alpha+\beta+n+1}{2})}}
{(1\mp q^{\frac{\alpha+1}{2}})(1\pm q^{\frac{\beta+1}{2}})}\right).
\label{spec2sv}
\end{eqnarray}
\end{cor}
\begin{proof}
The proof follows in exactly the same way as for the 
argument $\pm\frac12(q^\frac14+q^{-\frac14})$, except instead of using
the sum \cite[(4.23)]{AskeyWilson85},
directly use the quadratic transformation
for Askey--Wilson polynomials
\cite[(4.22)]{AskeyWilson85}. Then with this transformation, the resulting
terminating ${}_4\phi_3$ has a $q^{-1}$ as one of the numerator parameters,
so the sum truncates to the first two terms
and the above specializations follow. As well, simply
using Theorem \ref{thmTom} with $m=1$ in 
\eqref{secondPmxm} completes the proof.
\end{proof}
\begin{thm}\label{dualPmn}
Let $m,n,N\in\N_0$, $q\in\CCdag$, $\alpha,\beta\in\C$, and in the 
following special values for the 
continuous $q$-Jacobi polynomials for the positive 
sign in the argument, choose $\alpha=-N-1$ and leave 
$\beta$ unrestricted and for the negative sign in 
the argument chose 
$\beta=-N-1$ and $\alpha$ unrestricted. Furthermore, 
for \eqref{qJacRn} let $n\le N$, and for 
\eqref{qJacRm} let $m\le N$. Then
\begin{eqnarray}
&&\hspace{-1.5cm}P_n^{(\alpha,\beta)}(\pm\tfrac12(q^{\frac14
+\frac{m}{2}}\!+\!q^{-\frac14-\frac{m}{2}})|q)\nonumber\\
&&\hspace{0cm}=\label{qJacRn}
(\pm q^{\frac12\alpha})^n
\frac{
(\pm q^{\frac{\alpha+1}{2}},\mp q^{\frac{\beta+1}{2}};q)_n
}{(q^\frac12,-q^{\frac{\alpha+\beta+1}{2}};q)_n}
R_n(q^{-\frac{m}{2}}+q^{\frac{m+1}{2}};\pm q^{\frac12\alpha},
\pm q^{\frac12\beta},-1,-1|
q^\frac12)\\ &&\hspace{0cm}=\label{qJacRm}
(\pm q^{\frac12\alpha})^n
\frac{
(\pm q^{\frac{\alpha+1}{2}},\mp q^{\frac{\beta+1}{2}};q)_n}
{(q^\frac12,-q^{\frac{\alpha+\beta+1}{2}};q)_n}R_m(q^{-\frac{n}
{2}}+q^{\frac{\alpha+\beta+n+1}{2}};-1,-1,\pm q^{\frac12\alpha},
\pm q^{\frac12\beta}|q).
\end{eqnarray}
\end{thm}
\begin{proof}
Start with expression 
\eqref{qJacspec43}, and the definition
of the $q$-Racah polynomials \eqref{qRdefn}.
We choose
$\{\bar{\alpha},\bar{\beta},\bar{\gamma},
\bar{\delta}\}=\{\pm q^\alpha,\pm q^\beta,-1,-1\},$
and therefore $\mu(m)=q^{m+1}+q^{-m}$.
Since $\bar{\delta}=\bar{\gamma}=-1$, the third 
condition in \eqref{qRcond} is impossible, so we must use either 
one of the first two conditions. 
The only solution is for the positive argument to choose 
$\alpha=-N-1$ and leave $\beta$ 
unrestricted and for the negative argument to chose 
$\beta=-N-1$ and to leave $\alpha$ unrestricted. 
Next, consider \eqref{qJacRm}. 
We choose
$\{\bar{\alpha},\bar{\beta},\bar{\gamma},
\bar{\delta}\}
=\{-1,-1,\pm q^\alpha,\pm q^\beta\}$,
and therefore $\mu(n)=q^{\alpha+\beta+n+1}+q^{-n}$.

Since $\bar{\alpha}=\bar{\beta}=-1$, the first
condition in \eqref{qRcond} is impossible, so we must use 
either one of the second or third conditions. Again, the only solution 
is for the positive argument to choose $\alpha=-N-1$ 
and leave $\beta$ unrestricted and for the negative argument 
to chose $\beta=-N-1$ and to leave $\alpha$ unrestricted. 
Since these results are for
$P_n^{(\alpha,\beta)}(x|q^2)$, mapping $q^2\mapsto q$, completes the proof.
\end{proof}
\begin{cor}
If $\beta=-\alpha$ then one has the following duality relations 
for the specialized continuous $q$-Jacobi polynomials 
\begin{equation}
P_n^{(\alpha,-\alpha)}(\pm\tfrac12 (q^{\frac14+\frac{m}{2}}
+q^{-\frac14-\frac{m}{2}})|q),
\end{equation}
namely
\begin{equation}
R_n(q^{-\frac{m}{2}}+q^{\frac{m+1}{2}};\pm q^{\frac12\alpha},
\pm q^{-\frac12\alpha},-1,-1|q)
=
R_m(q^{-\frac{n}{2}}+q^{\frac{n+1}{2}};\pm q^{\frac12\alpha},
\pm q^{-\frac12\alpha},-1,-1|q).
\end{equation}
\end{cor}
\begin{proof}
Replace $\alpha+\beta=0$ in Theorem \ref{dualPmn} completes
the proof.
\end{proof}
\section{The Poisson kernel for continuous $q$-Jacobi polynomials}
In general, one may substitute special values
of an orthogonal polynomial into an identity for these
polynomials to obtain new specialized identities.
We will now focus on one such application. We now
utilize the special values which we 
derived in Section \ref{SpecctsqJac} within
the context of the Poisson kernel
for continuous $q$-Jacobi polynomials. This utilization
of special values for basic hypergeometric orthogonal
polynomials is both an application to the theory of 
basic hypergeometric orthogonal polynomials and 
nonterminating basic hypergeometric functions. 
The Poisson kernel for an orthogonal polynomial sequence 
is a bilinear generating function
and is a function of the two variables, 
$x=\frac12(z+z^{-1})$ and $y=\frac12(w+w^{-1})$ (as well 
as a power series parameter $|t|<1$ and other 
parameters involved). 
Inserting special values $z$, $w$ (which correspond 
to $x,y$) in a Poisson kernel 
results in the conversion of the bilinear generating 
function.
As we will see below, replacing either $z$ or $w$ in 
the Poisson kernel converts it to a generating function.
Furthermore, replacing both $z$ and $w$ in the Poisson 
kernel converts it to a transformation formula for an 
arbitrary argument (perhaps subject to certain constraints) 
nonterminating basic hypergeometric function.

\medskip

The most general Poisson kernel for Askey--Wilson polynomials
${\sf K}_{\,t}(x,y):={\sf K}_{\,t}(x,y;{\bf a}|q)$, 
which is given by
\begin{eqnarray}
&&\hspace{-0.45cm}{\sf K}_{\,t}(x,y)=\sum_{n=0}^\infty
\frac{(\frac{abcd}{q},\pm\sqrt{qabcd};q)_n\,p_n(x;{\bf a}|q)
p_n(y;{\bf a}|q)\,t^n}{(q,\pm\sqrt{\frac{abcd}{q}},ab,ac,ad,bc,bd,cd;q)_n}\\
&&\hspace{1cm}=\sum_{n=0}^\infty
\frac{(\frac{abcd}{q},\pm\sqrt{qabcd},ab,ac,ad;q)_nt^n}
{(q,\pm\sqrt{\frac{abcd}{q}},bc,bd,cd;q)_na^{2n}}
r_n(x;{\bf a}|q)
r_n(y;{\bf a}|q),
\end{eqnarray}
where we (and as well Gasper \& Rahman (1986)
\cite{GaspRah86}) have also used the normalized version of the 
Askey--Wilson polynomials 
\eqref{RAW}.
For the special case of the Askey--Wilson polynomials 
where $ad=bc$, 
then the Poisson kernel takes a more simplified form,
$K_{t}(x,y):={\sf K}_{\,t}(x,y;\{a,b,c,\frac{bc}{a}\}|q)$
namely
\begin{eqnarray}
&&K_{t}(x,y):=\sum_{n=0}^\infty
\frac{(\frac{b^2c^2}{q},\pm q^\frac12bc,ab,ac;q)_nt^n}{(q,\pm q^{-\frac12}bc,
\frac{b^2c}{a},\frac{bc^2}{a};q)_na^{2n}}
r_n(x;{\bf a}|q)
r_n(y;{\bf a}|q).
\end{eqnarray}
Gasper \& Rahman \cite[(6.13)]{GaspRah86} 
proved a very useful
form of this Poisson kernel for Askey--Wilson polynomials with parameters $(a,b,c,d)$ with the extra constraint $ad=bc$. 
This Poisson kernel is given 
in three terms,
each term given
as an infinite sum over a very-well poised and balanced 
${}_{10}W_9$.
One can see that for the continuous $q$-Jacobi 
polynomials \eqref{inter:AWqJac} 
with the particular choice of 
$\{a,b,c,d\}$ in \eqref{abcd}, the condition $ad=bc$
is satisfied.
Then the Poisson kernel for continuous
$q$-Jacobi polynomials is given by \cite[(2.10)]{GaspRah86}
\begin{equation}
K_t(x,y):=K_t^{(\alpha,\beta)}(x,y|q)=\sum_{n=0}^\infty \frac{(
q,q^{\alpha+\beta+1},
 q^{\frac{\alpha+\beta+3}{2}}
;q)_nt^n}
{(
q^{\alpha+1},q^{\beta+1},
 q^{\frac{\alpha+\beta+1}{2}}
;q)_nq^{(\alpha+\tfrac12)n}}P_n^{(\alpha,\beta)}(x|q)P_n^{(\alpha,\beta)}(y|q).
\label{PctsqJab}
\end{equation}
Replacing $\{a,b,c,d\}$ as in 
\eqref{abcd} produces the following form
for the Poisson kernel for continuous 
$q$-Jacobi polynomials.

{
\begin{thm} Let $q\in\CCdag$, $\alpha,\beta\in\CC$, 
$|t|<1$. Then the symmetric Poisson kernel for continuous $q$-Jacobi
polynomials is given by
\begin{eqnarray}
&&\hspace{-0.2cm}K_t(x,y)
\!=\!(1\!-\!t^2)\frac{(-q^{\frac{\alpha+\beta+4}{2}}t;q)_\infty}{(-q^{\frac{-\alpha-\beta-2}{2}}t;q)_\infty}\sum_{n=0}^\infty
\frac{(
q^{\frac{\alpha+\beta+2}{2}},
\pm q^{\frac{\alpha+\beta+3}{2}},
-q^{\frac12\beta\!+\!\frac34}z^\pm,-q^{\frac12\beta\!+\!\frac34}w^\pm;q)_nq^n}
{(q,q^{\beta+1},-q^{\frac12(\alpha+\beta+\topt{\scriptsize 2}{\scriptsize 3})},
-q^{\frac{\beta-\alpha+1}{2}},-q^{\frac{\alpha+\beta}{2}+2}t^{\pm};q)_n}\nonumber\\
&&\hspace{3.3cm}\times\Whyp{10}{9}{-q^{\frac{\alpha-\beta-2n-1}{2}}}
{q^{-n},-q^{\frac{-\alpha-\beta-2n-1}{2}},q^{-\beta-n},q^{\frac12\alpha+\frac14}
z^\pm,
q^{\frac12\alpha+\frac14}
w^\pm}
{q,q}\nonumber\\
&&\hspace{0.8cm}+\frac{(
q^{\alpha+\beta+2},t,
-q^{\frac{\alpha-\beta}{2}}t,
q^{\frac12\alpha+\frac34}z^\pm,
q^{\frac12\alpha+\frac14}w^\pm,
-q^{\frac12\beta+\frac14}tz^\pm,
-q^{\frac12\beta+\frac34}tw^\pm
;q)_\infty}
{(
q^{\alpha+1},
-q^{\frac{\alpha+\beta}{2}+\tops{1/2}{1}{3/2}},
-q^{\frac{\alpha-\beta}{2}},
q^{\beta+1}t,-q^{\frac{\alpha+\beta+2}{2}}t^{-1},tz^\pm w^\pm 
;q)_\infty}\nonumber\\
&&\hspace{2.0cm}\times\sum_{n=0}^\infty
\frac{(-t,\pm\sqrt{q}t,q^{\beta+1}t,tz^\pm w^\pm;q)_nq^n}
{(q,-q^{\frac{-\alpha-\beta}{2}}t,qt^2,-q^{\frac{\alpha-\beta}{2}}t,-q^{\frac12\beta+\frac14}tz^\pm,
-q^{\frac12\beta+\frac34}tw^\pm;q)_n}\nonumber\\
&&\hspace{3.3cm}\times\Whyp{10}{9}
{q^{\beta+n}t}
{q^nt,
-q^{\frac{\alpha+\beta+2n}{2}}t,
-q^{\frac{\beta-\alpha+2n}{2}}t,
-q^{\frac12\beta+\frac34}z^\pm,
-q^{\frac12\beta+\frac14}w^\pm
}{q,q}\nonumber\\
&&\hspace{0.8cm}+
\frac{(q^{\alpha+\beta+2},
t,-q^{\frac{\beta-\alpha}{2}}t,
-q^{\frac12\beta+\frac34}z^\pm,
-q^{\frac12\beta+\frac14}w^\pm,
q^{\frac12{\alpha}+\frac14}tz^\pm,
q^{\frac12{\alpha}+\frac34}tw^\pm
;q)_\infty}
{(
q^{\beta+1},
-q^{\frac{\alpha+\beta}{2}+\tops{1/2}{1}{3/2}},
-q^{\frac{\beta-\alpha}{2}},
q^{\alpha+1}t,
-q^{\frac12(\alpha+\beta+2)}t^{-1},
tz^\pm w^\pm
;q)_\infty}
\nonumber\\
&&\hspace{2.0cm}\times\sum_{n=0}^\infty
\frac{(-t,
\pm \sqrt{q}t,
q^{\alpha+1}
t, tz^\pm w^\pm 
;q)_nq^n}
{(q,-q^{\frac{-\alpha-\beta}{2}}t,
-q^{\frac{\beta-\alpha}{2}}t,qt^2,
q^{\frac12\alpha+\frac14}tz^\pm,
q^{\frac12\alpha+\frac34}tw^\pm;q)_n}\nonumber\\
&&\hspace{3.3cm}\times\Whyp{10}{9}{q^{\alpha+n}t}
{q^nt,-q^{\frac{\alpha+\beta+2n}{2}},
-q^{\frac{\alpha-\beta+2n}{2}}t,
q^{\frac12\alpha+\frac34}z^\pm,
q^{\frac12\alpha+\frac14}w^\pm}{q,q}.
\label{bigthm}
\end{eqnarray}
\end{thm}
}
\begin{proof}
{
This is obtained by making the replacement 
$\{a,b,c,d\}\mapsto \left\{q^{\frac12\alpha+
\frac14\topt{1}{3}},-q^{\frac12\beta+
\frac14\topt{1}{3}}\right\}$, namely \eqref{abcd},
into Gasper \& Rahman \cite[(6.13)]{GaspRah86}.}
\end{proof}
{
Note that there is an evident symmetry 
under the replacement $z\mapsto z^{-1}$ and $w\mapsto w^{-1}$ in \eqref{bigthm}.
Therefore there are 8 possibilities for substitutions, namely $z,w\in\{a,b,c,d\}$.
One simple application of the 
special values given in Corollary \ref{spec}, is 
that if you apply these substitutions to one of
the continuous $q$-Jacobi polynomials
in the bilinear generating function given by 
its Poisson kernel, it is converted to a 
generating function. Another easy application
of these special values is that if you apply 
them to both of the  continuous $q$-Jacobi 
polynomials, then the Poisson kernel is converted 
to a transformation formula for an arbitrary 
argument basic hypergeometric series. 
These kind of transformations are fairly rare 
in the literature of basic hypergeometric functions. 
The appearance of the arbitrary argument in 
the transformation formula comes from the power 
series parameter $t$, which appears in the 
bilinear generating function.
}

\subsection{Generating functions that condense 
from $K_t$}

{There is no evident symmetry in 
\eqref{bigthm} with respect to $a,b,c,d$, so there are many choices
of replacing with $z$ or $w$ using the special values  $a,b,c,d$ given in Corollary \ref{spec}.
Upon experimentation, we found that for two particular choices 
one obtains some interesting generating functions
for continuous $q$-Jacobi polynomials, which also have interesting $q\to1^{-}$ limits.}
\subsubsection{\boro{The $w=a$ generating function}}
{For the choice $w=a$ we obtain the following
interesting generating function for continuous $q$-Jacobi polynomials.}

{
\begin{thm}Let $q\in\CCdag$, $x=\frac12(z+z^{-1})\in\CC$, $\alpha,\beta,t\in\CCast$ such that 
$|q^{\alpha+\frac12}t|<1$. Then
\begin{eqnarray}
&&\hspace{-0.5cm} 
\sum_{n=0}^\infty
\frac{(q^{\alpha+\beta+1},q^{\frac12(\alpha+\beta+3)};q)_n}{(q^{\beta+1},q^{\frac12(\alpha+\beta+1)};q)_n}t^n
P_n^{(\alpha,\beta)}(x|q)\nonumber\\
&&\hspace{0.5cm}=\frac{(-q^{\frac12(3\alpha+\beta+5)}t,q^{2\alpha+1}t^2,-q^{\frac12(\alpha+\beta+2)}t,q^{\alpha+\frac12\beta+\frac74}tz^\pm;q)_\infty}
{(q^{2\alpha+2}t^2,-q^{\alpha+\beta+\frac52}t,-q^{\alpha+1}t,q^{\frac12\alpha+\frac14}tz^\pm;q)_\infty}\nonumber\\
&&\hspace{1.5cm}\times\Whyp87{-q^{\alpha+\beta+\frac32}t}{-q^{\alpha+\frac32}t,q^{\frac12(\beta-\alpha)},q^{\frac12(\alpha+\beta+3)},-q^{\frac12\beta+\frac34}z^\pm}{q,-q^{\alpha+\frac12}t}.
\label{JacPois}
\end{eqnarray}
\end{thm}
}
{
\begin{proof}
Start with \eqref{bigthm} and replacing $w=a$ in Corollary \ref{spec} and the result follows.
\end{proof}
}

{
This generating function has a $q\to 1^{-}$ limit given as follows.
\begin{thm}
Let $x,\alpha,\beta\in\CC$, $|t|<1$. Then one has the following generating functions for Jacobi polynomials, namely
\begin{eqnarray}
&&\hspace{-0.5cm} 
\sum_{n=0}^\infty
\frac{(\alpha+\beta+1)_n(\frac12(\alpha+\beta+3))_n}
{(\beta+1)_n(\frac12(\alpha+\beta+1))_n}t^n
P_n^{(\alpha,\beta)}(x)\nonumber\\
&&\hspace{0.5cm}=
\frac{1-t^2}{(1+t^2-2tx)^{\frac12(\alpha+\beta+3)}}
\hyp21{\frac12(\beta-\alpha),\frac12(\alpha+\beta+3)}{\beta+1}{\frac{-2t(x+1)}{1+t^2-2tx}}\nonumber\\
&&\hspace{0.5cm}=
\frac{2^{\frac12\beta}(1-t)\Gamma(\beta+1)}
{(t(1+x))^{\frac12\beta}(1+t^2-2tx)^{\frac12\alpha+1}}
P_{\alpha+1}^{-\beta}\left(\frac{1+t}{\sqrt{1+t^2-2tx}}\right),
\end{eqnarray}
where $P_\nu^\mu$ is an associated Legendre function of the first kind
\cite[(14.3.6)]{NIST:DLMF}.
\end{thm}
}
{
\begin{proof}
This generating function is obtained by
setting $w=a=q^{\frac12\alpha+\frac14}$ in
\eqref{JacPois} and then simplifying the resulting expression.
\end{proof}}

{As a special case, this generating function leads to the following generating function for continuous $q$-ultraspherical/Rogers polynomials.}
{
\begin{cor}
\label{firstRogersgf}
Let $q\in\CCdag$, $x=\frac12(z+z^{-1})\in\CC$, $\beta,t\in\CCast$ such that 
$|t|<1$. Then
\begin{equation}
\sum_{n=0}^\infty \frac{(q\beta;q)_n}{(\beta;q)_n}t^n C_n(x;\beta|q)=
(1\!-\!\beta t^2)
\frac{(
q\beta t z^\pm;q)_\infty}{
(
t z^\pm;q)_\infty}
\end{equation}
\end{cor}
}
{
\begin{proof}
Starting with \eqref{JacPois} with $\alpha=\beta$ causes the ${}_8W_7$ to become unity. Then using \cite[p.~473]{Koekoeketal}
\begin{equation}
P_n^{(\alpha,\alpha)}(x|q)=\frac{(q^{\alpha+1};q)_n}{(q^{2\alpha+1};q)_n}
q^{(\frac{\alpha}{2}+\frac14)n}
C_n(x;q^{\alpha+\frac12}|q),
\label{qJactoqGeg}
\end{equation}
completes the proof.
\end{proof}
}
{
\begin{rem}
The $q\to 1^{-}$ Gegenbauer polynomial limit of the generating function given by Corollary \ref{firstRogersgf} is
\begin{equation}
\sum_{n=0}^\infty \frac{(1+\beta)_n}{(\beta)_n}t^n C_n^\beta(x)=\sum_{n=0}^\infty \frac{\beta+n}{\beta}t^n C_n^\beta(x)=\frac{1-t^2}{(1+t^2-2tx)^{\beta+1}}.
\end{equation}
\end{rem}
}

\subsection{\boro{The $z=d$ generating function}}
{For the choice $z=d$ we obtain the following
interesting generating function for continuous $q$-Jacobi polynomials.}

{
\begin{thm}Let $q\in\CCdag$, $x=\frac12(z+z^{-1})\in\CC$, $\alpha,\beta,t\in\CCast$ such that 
$|q^{\alpha+\frac12}t|<1$. Then
\begin{eqnarray}
&&\hspace{-0.5cm} 
\sum_{n=0}^\infty
\frac{(q^{\alpha+\beta+1},\pm q^{\frac12(\alpha+\beta+3)};q)_n}{(q^{\alpha+1},\pm q^{\frac12(\alpha+\beta+1)};q)_n}t^n
P_n^{(\alpha,\beta)}(x|q)\nonumber\\
&&\hspace{0.5cm}=\frac{(q^{\alpha+\beta+3}t,q^{\alpha+\beta+2}t^2;q)_\infty}
{(t,q^{\alpha+\beta+3}t^2;q)_\infty}\qhyp54{q^{\frac12(\alpha+\beta+2)},\pm q^{\frac12(\alpha+\beta+3)},q^{\frac12\alpha+\frac14}z^\pm}{q^{\alpha+1},-q^{\frac12(\alpha+\beta+1)},q^{\alpha+\beta+3}t,\frac{q}{t}}{q,q}\nonumber\\
&&\hspace{0.8cm}+\frac{(q^{\alpha+\beta+2},-q^{\frac12(\alpha+\beta+1)}t,
-q^{\frac12(\alpha+\beta+2)}t,q^{\alpha+1}t,q^{\frac12\alpha+\frac14}z^\pm;q)_\infty}
{(q^{\alpha+1},-q^{\frac12(\alpha+\beta+1)},-q^{\frac12(\alpha+\beta+2)},\frac{1}{t},q^{\frac12\alpha+\frac14}tz^\pm;q)_\infty}\nonumber\\
&&\hspace{4.4cm}\times\qhyp54
{q^{\frac12(\alpha+\beta+2)}t,\pm q^{\frac12(\alpha+\beta+3)}t,q^{\frac12\alpha+\frac14}tz^\pm}
{-q^{\frac12(\alpha+\beta+1)}t,qt,q^{\alpha+1}t,q^{\alpha+\beta+3}t^2}
{q,q}.
\label{gfRog2}
\end{eqnarray}
\end{thm}
}
{
\begin{proof}
This generating function is obtained by
setting $z=d=-q^{\frac12\beta+\frac34}$ in
\eqref{JacPois}, replacing $y=\frac12(w+w^{-1})$ with $x$ and then simplifying the resulting expression completes the proof.
\end{proof}}

{
This generating function for continuous $q$-Jacobi polynomials has a $q\to 1^{-}$ limit given as follows.
\begin{thm}
Let $x,\alpha,\beta\in\CC$, $|t|<1$. Then one has the following generating function for Jacobi polynomials, namely
\begin{eqnarray}
&&\hspace{-0.5cm} 
\sum_{n=0}^\infty
\frac{(\alpha+\beta+1)_n(\frac12(\alpha+\beta+3))_n}{(\alpha+1)_n(\frac12(\alpha+\beta+1))_n}t^n
P_n^{(\alpha,\beta)}(x)\nonumber\\
&&\hspace{2cm}=\frac{1-t^2}{(1-t)^{\alpha+\beta+3}}
\hyp21{\frac12(\alpha+\beta+2),\frac12(\alpha+\beta+3)}{\alpha+1}{\frac{2t(y-1)}{(1-t)^2}}\nonumber\\
&&\hspace{2cm}=\frac{2^{\frac12\alpha}(1+t)\Gamma(\alpha+1)}{(t(y-1))^{\frac12\alpha}(1+t^2-2ty)^{\frac12\beta+1}}P_{\beta+1}^{-\alpha}\left(\frac{1-t}{\sqrt{1+t^2-2ty}}\right),
\end{eqnarray}
where $P_\nu^\mu$ is an associated Legendre function of the first kind
\cite[(14.3.6)]{NIST:DLMF}.
\end{thm}
}

{Starting from
\eqref{gfRog2}, using the
symmetric limit $\alpha=\beta$, we can produce 
a new generating function for continuous $q$-ultraspherical polynomials.
}
{
\begin{cor}
Let $q\in\CCdag$, $x=\frac12(z+z^{-1})\in\CC$, $\beta,t\in\CCast$ such that 
$|t|<1$. Then
\begin{eqnarray}
&&\hspace{-0.5cm}\sum_{n=0}^\infty\frac{(\pm q\beta;q)_n}{(\pm\beta;q)_n}t^n C_n(x;\beta|q)=
\frac{(q\beta,-q\beta tz^\pm,\beta t^2;q)_\infty}{(-\beta,t z^\pm,-q\beta t^2;q)_\infty}
\Whyp87{-\beta t^2}{-\frac{1}{q},\pm q\sqrt{\beta}t,tz^\pm}{q,q\beta}.
\end{eqnarray}
\end{cor}
}
{
\begin{proof}
Starting with \eqref{gfRog2} with $\alpha=\beta$ causes the sum of two ${}_5\phi_4(q,q)$s to become a sum of two ${}_4\phi_3(q,q)$s which then naturally transforms to an ${}_8W_7$ using
Bailey's transformation \cite[(17.9.16)]{NIST:DLMF}.
Then using \eqref{qJactoqGeg}
completes the proof.
\end{proof}
}

\subsubsection{{Arbitrary argument transformations that arise from Gasper \& Rahman's Poisson kernel for continuous $q$-Jacobi polynomials}}


{As mentioned just below \eqref{bigthm}, if one makes the double replacement of $z$ and $w$ in the Poisson kernel, one obtains an arbitrary argument transformation formula. Because of limitations on space in the current manuscript, we will only treat a single example. Note however, that if one considers the Poisson kernel for continuous $q$-Jacobi
polynomials and then one takes the special values \eqref{spec1or}--\eqref{spec4} 
for $z$ and $w$ simultaneously, then there are $4\times 4=16$ combinations. 
For each case, the Poisson kernel reduces to a single nonterminating basic 
hypergeometric series with an arbitrary argument. The
cases $z=w$ correspond to 4 unique transformations and the
12 off-diagonal combinations combine into pairs, with 
6 paired transformations for a single nonterminating basic hypergeometric series. 
Below we present the results for the computation of these transformations.
For the $z=w=a$ case, then the Poisson kernel
produces a well-poised ${}_3\phi_2$ with arbitrary
argument $z$. 
This is a well-poised nonterminating ${}_3\phi_2$ with
arbitrary argument $z$ expressed as a sum of two
${}_4\phi_3$s with argument $q$ or a
nonterminating ${}_8W_7$.
}
{
\begin{thm} Let $a,b,c\in\CC$ such that
$|a|,|z|<1$, $q\in\CCdag$. Then the 
following well-poised ${}_3\phi_2$ is
given by
\begin{eqnarray}
&&\hspace{-0.2cm}\qhyp32{ab,\sqrt{q}a,q\sqrt{ab}}{\sqrt{q}b,\sqrt{ab}}{q,z}
=\frac{(a^2z^2,-\sqrt{q^3a^3b}z
;q)_\infty}
{(qa^2z^2,-\sqrt{a/(qb)}z
;q)_\infty}
\qhyp43{q\sqrt{ab}
,\pm\sqrt{qab},-\sqrt{\frac{qb}{a}}
}{
\sqrt{q}b,
-\sqrt{q^3a^3b}z,-\sqrt{\frac{q^3b}{a}}z^{-1}
}{q,q}\nonumber\\
&&\hspace{2.0cm}+\frac{(qab,\sqrt{q}az,-\sqrt{qb/a},
-\sqrt{ab}z
;q)_\infty}
{(\sqrt{q}b,-q\sqrt{ab},z,-\sqrt{qb/a}z^{-1}
;q)_\infty}
\qhyp43{z,\pm az,-\sqrt{q}az}
{qa^2z^2,-\sqrt{\frac{qa}{b}}z,
-\sqrt{ab}z
}{q,q}\\
&&\hspace{0.5cm}=\frac{(a^2z^2,q\sqrt{ab}z,-\sqrt{qab}z,q\sqrt{a^3b}z,-\sqrt{q^3a^3b}z;q)_\infty}
{(qa^2z^2,z,az,-\sqrt{q}az,-\sqrt{q^3}abz;q)_\infty}\nonumber\\
&&\hspace{2cm}\times\Whyp{8}7{-\sqrt{q}abz}{q\sqrt{ab},-\sqrt{qab},\sqrt{\frac{b}{a}},-\sqrt{\frac{qb}{a}},-qaz}{q,-az}.
\end{eqnarray}
\end{thm}
}
{\begin{proof}
Start with the Poisson kernel for continuous
$q$-Jacobi polynomials and substitute
$z=w=a$. This converts the infinite series
into a single nonterminating ${}_3\phi_2$.
If one takes $q^{\alpha+\frac12}\mapsto a$, $q^{\beta+\frac12}\mapsto b$, then one arrives at a nonterminating 
transformation which is a sum of two
${}_4\phi_3$s with argument $q$. Using
Bailey's transformation \cite[(17.9.16)]{NIST:DLMF} we convert
it to a nonterminating ${}_8W_7$.
This completes the proof.
\end{proof}
}

\subsection*{Acknowledgements}
We would like to thank Tom Koornwinder, Mourad E.~H.~Ismail and George Gasper for valuable discussions.


\def\cprime{$'$} \def\dbar{\leavevmode\hbox to 0pt{\hskip.2ex \accent"16\hss}d}

\end{document}